\documentclass[11pt]{article}
\usepackage{graphicx}
\usepackage{amsmath}
\usepackage{color}
\usepackage[english]{babel}
\usepackage{amssymb,amsfonts,amsthm}
\usepackage{graphicx}
\usepackage{tikz}
\usepackage{tkz-graph}
\usepackage{tkz-berge}
\usetikzlibrary{arrows.meta,shapes.arrows}
\usepackage{tikz}
\usepackage{tikz,epsf}
\newif\ifpddf
\ifx\pddfoutput\undefined \pddffalse \else \pddfoutput=1 \pddftrue \fi
\ifpddf \else \fi \textwidth = 5.5in \textheight = 8.5 in
\newtheorem{theorem}{Theorem}

\newtheorem{corollary}[theorem]{Corollary}

\newtheorem{lemma}[theorem]{Lemma}

\newtheorem{observation}[theorem]{Observation}

\newtheorem{proposition}[theorem]{Proposition}

\begin{document}

\title{Independent double Roman domination in graphs}
\date{}
 \author{{\small Doost Ali Mojdeh}\thanks{Corresponding author}\ \ and\ \ {\small Zhila Mansouri} \\{\small Department of Mathematics,  University of Mazandaran,}\\ {\small Babolsar, Iran}\\
{\small $^*$damojdeh@umz.ac.ir}\\
{\small mansoury.zh@yahoo.com}
}

 \maketitle

\begin{abstract}
An independent double Roman dominating function (IDRDF) on a graph $G=(V,E)$ is a function $f:V(G)\rightarrow \{0,1,2,3\}$ having the property that if $f(v)=0$, then the vertex $v$ has at least two neighbors assigned $2$ under $f$ or one neighbor $w$ with assigned $3$ under $f$, and if $f(v)=1$, then there exists $w\in N(v)$ with $f(w)\geq2$ such that the positive weight vertices are independent. The weight of an IDRDF is the value $\sum_{u\in V}f(u)$. The independent double Roman domination number $i_{dR}(G)$ of a graph $G$ is the minimum weight of an IDRDF on G. We initiate
the study of the independent double Roman domination and show its relationships to both independent domination number (IDN) and
independent Roman $\{2\}$-domination number (IR2DN). We present several sharp bounds on the IDRDN of a graph $G$ in terms of the order of $G$, maximum degree and the minimum size of edge cover. Finally, we show that, any ordered pair $(a,b)$ is realizable as the IDN and IDRDN of some non-trivial tree if and only if $2a + 1 \le b \le 3a$.
\end{abstract}

\section{Introduction and terminologies}

\ \ \ Let $G = (V,E)$ be a simple graph with the  vertex set $V=V(G)$  and the edge set $E=E(G)$. For any vertex $ v\in V$, the \textit{open neighborhood} of $v$ is the set
$N(v)=\{u\in V|uv\in E\} $ and the \textit{closed neighborhood} of $v$ is the set $N[v] = N(v)\cup \{v\}$. For a set $S\subseteq V$, the open neighborhood of $S$ is $N (S)=\bigcup_{v\in S}N(v)$ and the closed neighborhood of $S$ is $N[S]=N(S)\cup S$. We use \cite{West} as a reference for terminology and notation which are not defined here.

\ \ \ Let $f$ be a function that assigns a subset of $\{1,2\}$ to each vertex of $G$, that is, $f:V(G)\rightarrow \mathcal{P}\{1,2\}$ where $\mathcal{P}\{1,2\}$ is the power set of $\{1,2\}$. If for each vertex $v\in V(G)$ such that $f(v)=\emptyset$, we have $\bigcup_{u\in N(v)} f(u)=\{1,2\}$, then $f$ is called a {\em $2$-rainbow dominating function} (2RDF) of $G$. The weight of a 2RDF $f$ is defined as $f(V(G))=\sum_{v\in V(G)}|f(v)|$. For simplicity, a 2RDF $f$ on a graph $G$ will be represented by the ordered partition $f=(V_\emptyset^f,V_{\{1\}}^f,V_{\{2\}} ^f,V_{\{1,2\}}^f)$ of $V(G)$ induced by $f$, where
$V_\emptyset^f=\{u\in V(G)|f(u)=\emptyset\}$, $V_{\{1\}}^f=\{u\in V(G)|f(u)=\{1\}\}$, $V_{\{2\}}^f=\{u\in V(G)|f(u)=\{2\}\}$ and $V_{\{1,2\}}^f=\{u\in V(G)|f(u)=\{1,2\}\}$. A function $f:V(G)\rightarrow \mathcal{P}\{1,2\}$ is called an {\em independent $2$-rainbow dominating function} (I2RDF) of $G$ if $f$ is a 2RDF and no two vertices in $V(G)\setminus V_\emptyset^f$ are adjacent. The {\em independent $2$-rainbow domination number} (I2RDN) $i_{r2}(G)$ is the minimum weight of an I2RDF of $G$ (see \cite{CJ}). The $2$-rainbow domination was introduced by Bresar et al. in \cite{Bresar}, and has been studied by several authors, for example, see \cite{Chang} and \cite{Wu}.

\ \ A function $f:V(G)\rightarrow \{0,1,2\}$ is a {\em Roman dominating function} (RDF) on $G$ if every vertex $u\in V$ for which $f(u)=0$ is adjacent to at least one vertex $v$ for which $f(v)=2$. The weight of a RDF is the value $f(V(G))=\sum_{v\in V(G)}f(v)$. The {\em Roman domination number} $\gamma_{R}(G)$ is the minimum weight of a RDF on $G$. The Roman domination was introduced by Cockayne et al. in \cite{Cockayne}. Since 2004, so many papers have been published on this topic, where several new variations were introduced: {\em weak Roman domination}, {\em maximal Roman domination}, {\em mixed Roman domination}, and  recently, Roman $\{2\}$-domination (\cite{Chellali}) and double Roman domination (\cite{Beeler}). A {\em Roman $\{2\}$-dominating function} (R2DF) is a function $f:V(G)\rightarrow \{0,1,2\}$ with the property that for every vertex $v\in V$ with $f(v)=0$, $f(N(v))\geq 2$, that is, there is a vertex $u\in N(v)$, with $f(u)=2$, or there are two vertices $x,y\in N(v$) with $f(x)=f(y)=1$. The weight of a R2DF is the value $f(V(G))= \sum_{v\in V(G)}f(v)$, and the minimum weight of a R2DF is called the {\em Roman $\{2\}$-domination number} and denoted by $\gamma_{\{R2\}}(G)$.

\ \ A {\em double Roman dominating function} (DRDF) on a graph $G$ is a function $f:V(G)\rightarrow \{0,1,2,3\}$ having the property that if $f(v)=0$, then the vertex $v$ has at least two neighbors assigned $2$ under $f$ or one neighbor $w$ with $f(w)=3$, and if $f(v)=1$, then there exists $w\in N(v)$ such that $f(w)\geq2$. The weight of a DRDF is the value $f(V(G))=\sum_{u\in V}f(u)$. The {\em double Roman domination number} (DRDN) $\gamma_{dR}(G)$ of a graph $G$ is the minimum weight of a DRDF on G.
For  simplicity, a DRDF $f$ on a graph $G$ may be represented by the ordered partition $f=(V_0,V_1,V_2,V_3)$ of $V(G)$ induced by $f$, where
$V_i=\{u\in V(G)|f(u)=i\}$ for $0\leq i\leq3$.
 A DRDF $f=(V_0,V_1,V_2,V_3)$ is called {\em independent} if $V_1\cup V_2\cup V_3$ is an independent set in $G$. The {\em independent double Roman domination number} (IDRDN) $i_{dR}(G)$ is the minimum weight of an independent double Roman dominating function (IDRDF) on $G$.

\ \ \ In this work, we mainly present lower and upper bounds on IDRDN of graphs, as for example by the well-known result of Gallai (concerning the maximum matching and the minimum edge cover) we prove that $i_{dR}(G)\leq i_{\{R2\}}(G)+\beta'(G)$ in which $G$ is a graph of order $n$ with no isolated vertices and $\beta'(G)$ is the maximum size of an edge cover of $G$. We also prove that $2i(T)+1\leq i_{dR}(T)\leq3i(T)$ for all trees $T$ of order $n\geq2$ and show that all values between the lower and upper bounds are realizable.


\section{Preliminary results}

\ \ In this section, we obtain some basic results and give the exact formulas for the IDRDNs for some well-known graphs. We first show $i_{dR}$ is well-defined for all graphs.

\begin{proposition}\label{prop1}
Every graph $G$ has an IDRDF.
\end{proposition}
\begin{proof}
Let $S$ be a maximal independent set of $G$. Then, every vertex in $V-S$ has at least one neighbor in $S$. Now the function
 $f:V(G)\rightarrow \{0,1,2,3\}$ which assigns $3$ to the vertices in $S$ and $0$ to the other ones is an IDRDF of $G$.
\end{proof}

\ \ In fact, Proposition \ref{prop1} guarantees that the IDRDF and therefore the IDRDN $i_{dR}(G)$ exists for all graphs $G$.

\ \ Since in any IDRDF $f=(V_0,V_1,V_2,V_3)$ the set $V_1\cup V_2\cup V_3$ is an independent set in $G$, so by the definition, $V_1=\emptyset$ for any $i_{dR}(G)$-function. It turns out to be useful in dealing with some results in this paper.

\begin{observation}\label{obser1}
In any IDRDF $f:V(G)\rightarrow \{0,1,2,3\}$,  $f(v)\neq1$ for all $v\in V(G)$.
\end{observation}

\ \ The DRDN of $P_n$ and $C_n$ were given in \cite{Ahangar}. The IDRDNs of $P_n$ and $C_n$ can be defined similar to those given in \cite{Ahangar} as follows.
\begin{proposition}\label{prop-2}
For $n\geq 1$,
\begin{equation*}
i_{dR}(P_{n})=\left\{\begin{array}{lll}
n&\mbox{if $n\equiv0$ \emph{(}mod\ 3\emph{)},}\vspace{1.5mm}\\
n+1&\mbox{otherwise.}
\end{array}
\right.
\end{equation*}
For $n\geq3$,
\begin{equation*}
i_{dR}(C_{n})=\left\{\begin{array}{lll}
n&\mbox{if $n\equiv0,2,3,4$ \emph{(}mod\ 6\emph{)},}\vspace{1.5mm}\\
n+1&\mbox{otherwise.}
\end{array}
\right.
\end{equation*}
\end{proposition}
\begin{proof}
Consider the path $v_1\cdots v_n$. It is easy to see that the function  $f:V(G)\rightarrow \{0,2,3\}$ defined by $f(v_{3i+2})=3$ and $f(v_j)=0$ for other vertices if $n=0$ (mod $3$), and $f(v_{3i+2})=f(v_n)=3$ and $f(v_j)=0$ for the other vertices if $n=1,2 $ (mod $3$) is an IDRDF of $P_n$ with the weight $\gamma_{dR}(P_n)$. Since $\gamma_{dR}(P_n)\leq i_{dR}(P_n)$, it follows that $i_{dR}(P_n)=\gamma_{dR}(P_n)$.\\
For the cycle $C_n$, if we assign $2$ to the vertices with the even index and $0$ to others when $n$ is even,  if $n\equiv 3$ (mod $6$), then we assign $3$ to the vertices with the index $3i$ and $0$ to the others, if $n\equiv1$ (mod $6$) we assign $3$ to the vertices $v_{3i}$, $2$ to $v_1$ and $0$ to others,  and finally if $n\equiv 5$ (mod $6$), then we assign $3$ to the vertices $v_{3i}$ and $v_1$, and $0$ to others.
We also have $i_{dR}(C_n)=\gamma_{dR}(C_n)$.
\end{proof}

In what follows the IDRDNs of the complete graphs and complete $r(\geq2)$-partite graphs are given.

\begin{observation}\label{OBSE}
\emph{(i)} Let $G=K_{m_1,\cdots,m_r}$ be a complete $r(\geq2)$-partite  with size $m_1\leq\cdots\leq m_r$. Then,
\begin{equation*}
i_{dR}(G)=\left\{\begin{array}{lll}
3&\mbox{if\ $m_1=1$,}\vspace{1.5mm}\\
2 m_1 &\mbox{otherwise.}
\end{array}
\right.
\end{equation*}
\emph{(ii)} $i_{dR}(K_n)=3.$\vspace{1mm}

\emph{(iii)} $i_{dR}(G)=3$ if and only if $\Delta(G)=n-1$.
\end{observation}


\section{Independent double Roman and Independent Roman $\{2\}$-domination}

\ \ In this section, we establish some relationships between the IDRDN and IR2DN in graphs.

\begin{proposition}\label{prop+} For any graph $G$,
$$\frac{3}{2}i_{\{R2\}}(G)\leq i_{dR}(G) \le 2i_{\{R2\}}(G)$$
and these bounds are sharp.
\end{proposition}
\begin{proof}
Let $f=(V_0,V_1,V_2)$ be an $i_{\{R2\}}(G)$-function with $i_{\{R2\}}(G)=|V_1|+2|V_2|$. Then, $g=(V'_0=V_0,V'_2=V_1,V'_3=V_2)$ is an IDRDF of $G$.
Therefore $i_{dR}(G) \le 2|V_1|+3|V_2|\le 2(|V_1|+2|V_2|)=2i_{\{R2\}}(G)$.\\
The graph $\overline{K_n}$ and a graph of order at least $4$ with two independent vertices $u$ and $v$ such that all the other vertices are adjacent to both $u$ and $v$ are graphs that achieve the upper bound.

\ \ \ In order to prove the lower bound we let $f=(V_0^f,V_2^f,V_3^f)$ be an $i_{dR}(G)$-function. If $g=(V'_0=V_0^f,V'_1=V_2^f, V'_2=V_3^f)$, then $g$ is an IR2DF of $G$ with $w(g)=|V_2^f|+2|V_3^f| \le\frac{2}{3}(2|V_2^f|+3|V_3^f|)=\frac{2}{3}i_{dR}(G)$.

Using Part (iii) of Observation \ref{OBSE}, we have $\frac{3}{2}i_{\{R2\}}(G)=i_{dR}(G)=3$ for all graphs $G$ with $\Delta(G)=n-1$. So, the lower bound is sharp.
\end{proof}

As an immediate result we have,
\begin{corollary}
For every graph $G$, $i_{\{R2\}}(G) < i_{dR}(G)$.
\end{corollary}

\begin{proposition}\label{prop4}
For any graphs $G$, $i_{dR}(G)\leq 2i_{r2}(G)$. This bound is sharp.
\end{proposition}
\begin{proof}
Let $f=(V_\emptyset^f,V_{\{1\}}^f,V_{\{2\}}^f,V_{\{1,2\}}^f)$ be an $i_{r2}(G)$-function. Define the function $g$ on $G$ by $g(x)=3$ if $x\in V_{\{1,2\}}^f$, $g(x)=2$ if $x\in V_{\{1\}}^f\cup V_{\{2\}}^f$, and $g(x)=0$ if $x\in V_\emptyset^f$. Clearly $g$ is an IDRDF of $G$ with the weight $i_{r2}(G)$, and therefore
$$i_{dR}(G)\leq w(g)=3|V_{\{1,2\}}^f|+2|V_{\{1\}}^f|+2|V_{\{2\}}^f|$$
$$=2(2|V_{\{1,2\}}^f|+|V_{\{1\}}^f|+|V_{\{2\}}^f|)-|V_{\{1,2\}}^f|$$
$$=2i_{r2}(G)-|V_{\{1,2\}}^f|\leq 2i_{r2}(G).$$
The upper bound holds with the equality for the complete $r(\geq 2)$-partite graphs $G=K_{m_1,\cdots,m_r}$ with $m_1\leq\cdots\leq m_r$, where $(m_1\ge 2)$.
\end{proof}

\begin{theorem}\label{theo7}
Let $G$ be a connected graph. Then,
$$i_{dR}(G)\geq i_{\{R2\}}(G)+i(G)$$
and this bound is sharp.
\end{theorem}
\begin{proof}

The bound clearly holds for $K_1$ and $K_2$. So, we assume that $G$ is of order $n\geq 3$. In view of Observation \ref{obser1} we let
$f=(V_0^f,V_2^f,V_3^f)$ be an $i_{dR}(G)$-function. We observe that $V_2^f\cup V_3^f$ is an independent dominating set in $G$ and therefore,
\begin{equation}\label{EQ1}
i(G)\leq |V_2^f|+|V_3^f|.
\end{equation}
We define $g:V(G)\rightarrow \{0,1,2\}$ by
 $$g(v)=
\begin{cases}
 f(v)-1 & if\ v\in V_2^f\cup V_3^f \\
 f(v) & if\ v\in V_{0}^f.
\end{cases}$$
 Therefore, $g$ is an IR2DF. Moreover,
\begin{equation}\label{EQ2}
i_{\{R2\}}(G)\leq \sum_{v\in V(G)}g(v)=|v_2^f|+2|V_3^f|.
\end{equation}
Together inequalities (\ref{EQ1}) and (\ref{EQ2}) imply that $i(G)+i_{dR}(G)\leq2|V_2^f|+3|V_3^f|=i_{dR}(G)$.

\ \ To see the bound is sharp it suffices to consider the complete graph $K_n$ or any graph $G$ with $\Delta(G)=|V(G)|-1$ or the $r$-partite graph $G=K_{m_1,\cdots,m_r}$ with $m_1\leq\cdots\leq m_r$, where $m_1\ge 2$.
\end{proof}

\ \ We recall that a matching $M$ of graph $G$ is a subset of the set of edges $E(G)$, such that no vertex in $V(G)$ is incident to more than one edge in $M$ in the other words one can say that no two edges in $M$ have a common vertex. A matching $M$ is said to be maximum if, $|M|\geq|M'|$, for any other matching $M'$ of $G$. We also recall that an edge cover $Q$ of a graph $G$ is a set of edges such that each vertex in $G$ is incident to at least one edge in $Q$. By $\alpha'(G)$ and $\beta'(G)$ we mean the maximum cardinality of a matching and the minimum cardinality of an edge cover in $G$, respectively.

\ \ We make use of the following classic result due to Gallai.

\begin{lemma}\emph{(\cite{West})}\label{A}
 If $G$ is a graph of order $n$ and with no isolated vertices, then $\alpha'(G)+\beta'(G)=n$.
\end{lemma}

\begin{theorem}
Let $G$ be a graph with no isolated vertices. Then,
$$i_{dR}(G)\leq i_{\{R2\}}(G)+\beta'(G)$$
and this bound is sharp.
\end{theorem}
\begin{proof}
Let $f:V(G)\rightarrow \{0,1,2\}$ be an $i_{\{R2\}}(G)$-function. We define $g:V(G)\rightarrow \{0,2,3\}$ by $$g(v)=
\begin{cases}
 f(v)+1 & \mbox{if }  v\in V_1^f\cup V_2^f \\
 0 & \mbox{if }  v\in V_{0}^f.
\end{cases}$$
This equation shows that $g$ is an IDRDF of $G$. Therefore,
\begin{equation}\label{EQ3}
i_{dR}(G)\leq \sum_{v\in V(G)}g(v)=2|V_1^f|+3|V_2^f|=i_{\{R2\}}(G)+|V_1^f|+|V_2^f|
=i_{\{R2\}}(G)+n-|V_0^f|.
\end{equation}
It is easy to see that at least one vertex incident to each edge of a maximum matching belongs to $V_0^f$. Therefore, $\alpha'(G)\leq|V_0^f|$. Now Lemma \ref{A} and  the inequality (\ref{EQ3}) imply that
$$i_{dR}(G)\leq i_{\{R2\}}(G)+n-\alpha'(G)=i_{\{R2\}}(G)+\beta'(G).$$
For sharpness consider the complete bipartite graph $K_{p,p}$ in which $p\geq 2$. Then, $i_{dR}(K_{p,p})=2p=p+p=i_{\{R2\}}(K_{p,p})+\beta'(K_{p,p})$.
\end{proof}


\section{Independent double Roman and independent domination}

We first give some lower and upper bounds on the IDRDN in terms of the
independent domination number.

\begin{proposition}\label{prop-3}
For any graph $G$, $2i(G)\leq i_{dR}(G)\leq 3i(G)$. These bounds are sharp.
\end{proposition}
\begin{proof}
For the lower bound, in view of Observation \ref{obser1} we let $f=(V_0,V_2,V_3)$ be an $i_{dR}(G)$-function. Let $S$ be a minimum independent dominating set in $G$. Note that $(\emptyset,\emptyset,S)$ is an IDRDF. This yields the upper bound $i_{dR}(G)\leq 3i(G)$.\\
Furthermore, $V_2\cup V_3$ is an independent dominating set in $G$. Thus, $i(G)\leq |V_2|+|V_3|$. Taking into account this, we obtain the lower bound as follows.
$$i_{dR}(G)=2|V_2|+3|V_3|\geq 2(|V_2|+|V_3|)\geq 2i(G).$$
For the upper bound, let $\mathcal{F}$ be the family of graphs $G$ with $\Delta(G)=|V(G)|-1$. Then $i(G)=1$ and $i_{dR}(G)=3$. For the lower bound, let $\mathcal{H}$ be a family of graphs $G$ such that $G$ has a minimum maximal independent set with at least $2$ vertices like $U$ such that every vertices of $G-U$ is adjacent to $2$ vertices of $U$. Then $i_{dR}(G)=2i(G)$. For example, let
$G=K_{m_1,\cdots,m_r}$ be a complete $r(\geq 2)$-partite graph with $m_1\leq\cdots\leq m_r$ where $m_1\ge 2$.
\end{proof}

Recall that a set $R\subseteq V(G)$ is a packing set of $G$ if $N[x]\cap N[y]=\emptyset$ holds for any two distinct vertices $x,y\in R$. The packing number $\rho(G)$ is the maximum cardinality of a packing set in $G$. Let $\delta$ denote the minimum degree of the graph $G$. A classical result shows that: for any graph $G$, $\rho(G)\leq \gamma(G) \le i(G)$.

\begin{proposition}\label{prop10}
If $G$ is a connected graph of order $n$, then
$$i_{dR}(G)+(2\delta-1)\rho(G)\leq 2n$$
and this bound is sharp.
\end{proposition}
\begin{proof}
Let $R$ be a maximum packing set of $G$ and $A=N(R)$. Let $B=V(G)-(A\cup R)$. Each vertex in $A$ has exactly one neighbor in $R$ and each vertex in $R$ has at least $\delta$ neighbors in $A$. Therefore,
$\delta|R|\leq|[R,V-R]|=|A|$. Therefore,
$$|B|=n-|A\cup R|=n-|A|-|R|\leq n-(\delta+1)|R|$$
Now we define $f:V(G)\rightarrow \{0,1,2,3\}$ by,
$$f(v)=
\begin{cases}
  3& \mbox{if } v\in R\\
  0 & \mbox{if } v\in A \\
  2& \mbox{if } v\in B.
\end{cases}$$
It is easy to see that $f$ is an IDRDF of $G$. Therefore,
$$i_{dR}(G)\leq w(f)=3|R|+2|B|=3|R|+2n-(2\delta+2)|R|=2n-(2\delta-1)\rho(G).$$
This bound is sharp for the complete graph $K_n$, for $n\geq 2$.
\end{proof}

\begin{proposition}\label{prop6}
For any graph $G$ of order $n$ with maximum degree $\Delta$,
$$i_{dR}(G)\geq \frac{2n}{\Delta}+\frac{\Delta-2}{\Delta}i(G)$$
and this bound is sharp.
\end{proposition}
\begin{proof}
Let $f=(V_0^f, V_2^f,V_3^f)$ be an $i_{dR}(G)$-function of $G$. Using Observation \ref{obser1} we may assume that $V_1^f=\emptyset$. Let $S=V_0^f\bigcap N(V_3)$ and $T=V_0^f\bigcap N(V_2)$. Since each vertex in $V_3^f$ dominates at most $\Delta$ vertices of $S$, we have $|S|\leq \Delta|V_3^f|$. Since each vertex in $V_2^f$ dominates at most $\Delta$ vertices of $T$ and since each vertex in $T$ has at least two neighbors in $T$, we have
$2|T|\leq |E(V_2,T)|\leq\Delta|V_2^f|$ yielding $|T|\leq \frac{\Delta}{2}|V_2^f|$. Hence, $|V_0^f|=|S|+|T|\leq \Delta|V_3^f|+\frac{\Delta}{2}|V_2^f|$. Now we have
$$\Delta i_{dR}(G)=\Delta(2|V_2^f|+3|V_3^f|)$$
$$=\Delta(|V_2^f|+|V_3^f|)+2\Delta|V_3^f|+\Delta|V_2^f|$$
$$\geq\Delta(|V_2^f|+|V_3^f|)+2|V_0^f|$$
$$=(\Delta-2)(|V_2^f|+|V_3^f|)+2n.$$
Since $V_2^f\cup V_3^f$ is an independent dominating set of $G$, it follows that $\Delta i_{dR}(G)\geq (\Delta-2)i(G)+2n$.

\ \ The bound is sharp for the  cycles $C_n$ where $n\equiv 0,2,3,4\ (mod\ 6)$ and for the paths $P_n$ where $n\equiv 0\ (mod\ 3)$.
\end{proof}


\section{Trees}

We make use the following result to show that the IDRDNs of trees are bounded from below and above just in terms of the independent domination number. The result  may be important in its own right.

\begin{theorem}\label{the12}
For any tree $T$ of order $n\geq 2$, $i_{\{R2\}}(T)\geq i(T)+1$.
\end{theorem}
\begin{proof}
If $T$ is a star, then $i_{\{R2\}}(T)=2=i(T)+1$. Let $T$ be a double star $T_{r,s}$ in which $1\leq r\leq s$. If $r=1$, then $i_{\{R2\}}(T)=3=2+1=i(T)+1$. If $r\geq2$, then $i_{\{R2\}}(T)=2+r=i(T)+1$. So, we assume from now on that $diam(T)\geq4$. Let $P$ be a diametral $r,s$-path of $T$. We root the tree $T$ at $r$. Let $f$ be an $i_{\{R2\}}(T)$-function of $T$. We now deal with two cases depending on $f$.

{\em Case 1}. Suppose that there exists a vertex $x$ for which $f(x)=2$. It is easy to observe that $S=\{v\in V(T)|f(v)\neq0\}$ is an independent dominating set in $T$. Therefore, $i(T)\leq |S|\leq w(f)-1=i_{\{R2\}}(T)-1$.\vspace{1mm}

{\em Case 2}. Suppose that $f(x)=0$ or $1$ for all $x\in V(T)$. Since $f$ is an IR2DF, it follows that $f$ assigns $1$ to all leaves and $0$ to all support vertices. Let $u$ be a support vertex and $L_u$ be the set of all leaves adjacent to $u$. If  $x\in N(u)-L_u$ has the weight $f(x)=1$, then every vertex $v\in N(x)-\{u\}$ has a neighbor $w\ne u$ with $f(w)=1$ and $f(v)=0$ by the properties of the IR2DF $f$. Let $A= \{x\in N(u)-L_u:  f(x)=1\}$. Since $f(u)=0$, $|L_u \cup A|\ge 2$.
If now we define $f':V(G)\rightarrow \{0,1\}$ by
$$f'(z)=
\begin{cases}
  0& \mbox{if } z\in (L_u \cup A \cup V_0^f\setminus \{u\})\\
  1& otherwise,
\end{cases}$$
it follows that $S'=\{v\in V(T): f'(v)\neq0\}$ is an independent dominating set of $T$. Thus, $i(T)\leq |S'|\leq w(f')\leq w(f)-1=i_{\{R2\}}(T)-1$.
\end{proof}

In what follows, for the sake of completeness, we characterize the familly of all trees for which the lower bound in Theorem \ref{the12} holds with equality. To this aim, we begin with the following lemma. 

\begin{lemma}\label{lemma111}
Let $T$ be a tree and $f=(V_0, V_1,V_2)$ be a $\gamma_{\{R2\}}(T)$-function. If $w(f)=\gamma_{\{R2\}}(T)=\gamma(T)+1$, then $V_1\cup V_2$ is an independent set.
\end{lemma}

\begin{proof}
Since $f=(V_0, V_1,V_2)$ be a $\gamma_{\{R2\}}(T)$-function, it follows that $|V_1|+2|V_2|=\gamma_{\{R2\}}(T)=\gamma(T)+1\le |V_1|+|V_2|+1$. Therefore, $|V_2|\le 1$. Let two vertices $v$ and $u$ in $V_1\cup V_2$ be adjacent. Since $T$ is a tree, no vertex in $V_0$ is adjacent to both $u$ and $v$.
Let $u \in V_1$ and $v \in V_2$. Then each vertex $w$ in $V_0\cap N(u)$ has another neighbor in $V_1$. Therefore, $V_1\setminus \{u\}\cup \{v\}$ is a dominating set in $T$ of cardinality $|V_1|< |V_1|+2|V_2|-1=\gamma(T)$, a contradiction. Let $\{u ,v\} \subseteq V_1$. If $z\in N(u)\cap V_1$ and $z\ne v$, then $V_1\setminus \{z,v\} \cup V_2$ is a dominating set in $T$ of cardinality $|V_1|< |V_1|+2|V_2|-1=\gamma(T)$, a contradiction. Similarly, for $z\in N(v)\cap V_1$ and $z\ne u$ we achieve the same contradiction. If every vertex $z \in N(u)\cup N(v)\{u,v\}$ is in $V_0$, then there are $z_1\in N(u)\cap V_0$ and $z_2\in N(v)\cap V_0$ such that $z_1$ has  a neighbor  $w_1$ other than $u$ and $z_2$ has  a neighbor $w_2$ other than $v$ with positive weights. Now if $f(w_1)=f(w_2)=1$, then the set $V_1\setminus \{u,v,w_1,w_2\} \cup \{z_1,z_2\}\cup V_2$ is a dominating set of cardinality $\gamma_{\{R2\}}(T)-2$, a contradiction. If $f(w_1)=1$ and $f(w_2)=2$, then the set $V_1\setminus \{u,v,w_1\} \cup \{z_1,z_2\}\cup V_2$ is a dominating set of size $\gamma_{\{R2\}}(T)-2$, which is again a contradiction. Thus $V_1\cup V_2$ is an independent set.
\end{proof}

From Lemma \ref{lemma111}, we have the following.

\begin{corollary} \label{cor5}
Let $T$ be a tree and $f=(V_0, V_1,V_2)$ be a $\gamma_{\{R2\}}(T)$-function. If $\gamma_{\{R2\}}(T)=\gamma(T)+1$, then $i_{\{R2\}}(T)=\gamma_{\{R2\}}(T)$ and $i(T)=\gamma(T)$.
\end{corollary}
\begin{proof}
In Lemma \ref{lemma111}, it has been shown that $V_1\cup V_2$ is independent. If $V_2\ne \emptyset$, then  $V_1\cup V_2$ is an independent dominating set and so
$i_{\{R2\}}(T)=\gamma_{\{R2\}}(T)$. If $V_2= \emptyset$, then there exists a vertex $v\in V_0$ for which $v$ has exactly two neighbors in $V_1$ like $u,w$.
We define the function $g: V \rightarrow\{0,1\}$ by
$$g(z)=
\begin{cases}
  0& \mbox{if } z\in (V_0 \cup \{v\} \setminus \{u,w\})\\
  1& otherwise.
\end{cases}$$
It follows that $S=\{v\in V(T): g(v)=1\}$ is a minimum dominating set of $T$ of cardinality $\gamma_{\{R2\}}(T)-1$. Since $S$ is an independent set, it also follows that $|S|=i(T)$.
\end{proof}

In \cite{hk}, Henning and Klostermeyer characterized all trees $T$ of order $n\geq2$ for which $\gamma_{\{R2\}}(T)=\gamma(T)+1$. To this aim, they introduced two families of trees.\\
For positive integers $r$ and $s$, let $F_{r,s}$ be the tree obtained from a double star $S_{r,s}$ by subdividing every edge exactly once.
For example, $P_7 = F_{1,1}$. The tree $F_{4,4}$ is shown in Figure 1. Let $\mathfrak{F}$ be the family of all such trees $F_{r,s}$, that is, $\mathfrak{F} = \{F_{r,s} : r, s \ge 1\}$.\\

\begin{figure}[htb]
\tikzstyle{every node}=[circle, draw, fill=black!0, inner sep=0pt,minimum width=.16cm]
\begin{center}
\begin{tikzpicture}[thick,scale=.6]
  \draw(0,0) { 
    +(-5.00,0.00) -- +(-4.00,0.00)
    +(-5.00,1.00) -- +(-4.00,1.00)
    +(-5.00,2.00) -- +(-4.00,2.00)
    +(-5.00,3.00) -- +(-4.00,3.00)
    +(-3.00,1.500) -- +(-4.00,0.00)
    +(-3.00,1.500) -- +(-4.00,1.00)
    +(-3.00,1.500) -- +(-4.00,2.00)
    +(-3.00,1.500) -- +(-4.00,3.00)

    +(-3.00,1.500) -- +(-2.00,1.500)
    +(-2.00,1.500) -- +(-1.00,1.500)

    +(1.00,0.00) -- +(0.00,0.00)
    +(1.00,1.00) -- +(0.00,1.00)
    +(1.00,2.00) -- +(0.00,2.00)
    +(1.00,3.00) -- +(0.00,3.00)
    +(-1.00,1.500) -- +(0.00,0.00)
    +(-1.00,1.500) -- +(0.00,1.00)
    +(-1.00,1.500) -- +(0.00,2.00)
    +(-1.00,1.500) -- +(0.00,3.00)
    +(-5.00,0.00) node{}
    +(-5.00,1.00) node{}
    +(-5.00,2.00) node{}
    +(-5.00,3.00) node{}
    +(-4.00,0.00) node{}
    +(-4.00,1.00) node{}
    +(-4.00,2.00) node{}
    +(-4.00,3.00) node{}
    +(-3.00,1.500) node{}
    +(-2.00,1.500) node{}
    +(-1.00,1.500) node{}

    +(1.00,0.00) node{}
    +(1.00,1.00) node{}
    +(1.00,2.00) node{}
    +(1.00,3.00) node{}
    +(0.00,0.00) node{}
    +(0.00,1.00) node{}
    +(0.00,2.00) node{}
    +(0.00,3.00) node{}

  };
\end{tikzpicture}
\end{center}
\vskip -0.6 cm \caption{A subdivided double star.} \label{f:F5}
\end{figure}
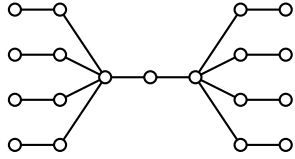

Let $\mathcal{T}$ be the family of trees $T_{k,j}$ of order $k \ge 2$ where $k \ge 2j +1$ and $j \ge 0$, obtained from a star by subdividing $j$ edges
exactly once. The tree $T_{12,4}$ is shown in Figure 2.

\begin{figure}[htb]
\tikzstyle{every node}=[circle, draw, fill=black!0, inner sep=0pt,minimum width=.16cm]
\begin{center}
\begin{tikzpicture}[thick,scale=.6]
  \draw(0,0) { 
    +(3.70,2.50) -- +(1.20,1.00)
    +(3.70,2.50) -- +(2.20,1.00)
    +(3.70,2.50) -- +(3.20,1.00)
    +(3.20,1.00) -- +(3.20,0.00)
    +(2.20,1.00) -- +(2.20,0.00)
    +(1.20,1.00) -- +(1.20,0.00)
    +(3.70,2.50) -- +(0.20,1.00)
    +(0.20,1.00) -- +(0.20,0.00)
    +(3.70,2.50) -- +(4.20,1.00)
    +(3.70,2.50) -- +(5.20,1.00)
    +(3.70,2.50) -- +(6.20,1.00)
    +(3.70,2.50) node{}
    +(0.20,1.00) node{}
    +(1.20,1.00) node{}
    +(2.20,1.00) node{}
    +(3.20,1.00) node{}
    +(4.20,1.00) node{}
    +(5.20,1.00) node{}
    +(6.20,1.00) node{}
    +(0.20,0.00) node{}
    +(1.20,0.00) node{}
    +(2.20,0.00) node{}
    +(3.20,0.00) node{}

  };
\end{tikzpicture}
\end{center}
\vskip -0.6 cm \caption{A subdivided star.} \label{f:F5}
\end{figure}
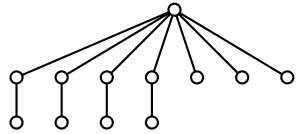

They proved that;

\begin{theorem} \emph{(\cite{hk}, \label{theo-henning} Theorem 6)} Let $T$ be a non-trivial tree. Then, $$\gamma_{\{R2\}}(T)=\gamma(T)+1\ \emph{if}\ \emph{and}\ \emph{only}\ \emph{if}\ T\in \mathcal{T} \cup \mathfrak{F}.$$
\end{theorem}

Now we characterize all trees $T$ of order $n\geq2$ for which $i_{\{R2\}}(T)=i(T)+1$. We deduce this result
from Lemma \ref{lemma111}, Corollary \ref{cor5} and Theorem \ref{theo-henning}.

\begin{theorem} \label{the13} Let $T$ be a non-trivial tree. Then, $$i_{\{R2\}}(T)=i(T)+1\ \emph{if}\ \emph{and}\ \emph{only}\ \emph{if}\ T\in \mathcal{T} \cup \mathfrak{F}.$$
\end{theorem}

Theorem \ref{theo7}, Proposition \ref{prop-3} and Theorem \ref{the12} yield the following for trees.

\begin{corollary}\label{cor6}
For any tree $T$ of order $n\geq 2$, $2i(T)+1\leq i_{dR}(T)\leq3i(T)$.
\end{corollary}

Our final result in this section shows that every value in the range of Corollary \ref{cor6} is realizable for trees, that is, all values between the lower and upper bounds of Corollary \ref{cor6} are realizable. We first recall that the {\em corona} $G\circ K_1$ of a graph $G$ is formed from $G$ by adding a new vertex $w$ and edge $vw$ for each vertex $v \in V(G)$.

\begin{theorem}
An ordered pair $(a,b)$ is realizable as the IDN and IDRDN of some
non-trivial tree if and only if $2a + 1 \le b \le 3a$.
\end{theorem}

\begin{proof}
Let $T$ be a tree with $i (T)=a$ and $i_{dR}(T)=b$. By Corollary \ref{cor6}, $2a+1 \le b \le 3a$.
Next we show that each ordered pair is realizable. For $b=2a+1$, consider the corona of the star $K_{1,t}$, for $t \ge 1$. We assign the value $1$ to all support vertices other than the center and also assign value $1$ to the leaf neighbor of the center. It is
straightforward to check that $i(K_{1,t} \circ K_1) = t + 1$. For the IDRDN we assign the value $2$ to all leaves other than the leaf neighbor of the center and the value $3$ to the center. It is easy to see that $i_{dR}(K_{1,t} \circ K_1) = 2t + 3 = 2i(K_{1,t} \circ K_1) + 1$.
Assume now that $b\ge 2a +2$. Let $T$ be the tree formed from a subdivided star
$K^*_{1,a}$  by choosing $b - (2a + 2)$ support
vertices of $K^*_{1,a}$ and adding another vertex as a neighbor of leaf of each of them. Thus, $T$ has $b-2a-2$ vertices other than the center which are neither support vertices nor leaves. Again,
it is straightforward to check that $i(T)=a$ (the set of support vertices form an $i(T)$-set). To see that $i_{dR}(T ) = b$, note
that each of the $b-2a-2$ new support vertices must be assigned a value $3$ under any $i_{dR}(T)$-function. Assigning a value $2$ to the leaf non-adjacent to the new support vertices  and the center of $T$ and $0$ to the other vertices. It is simple to check that this function is in
fact of the minimum weight. Hence, we have $i_{dR}(T) = 3(b-2a-2)+2(a-(b-(2a+2)))+2 =b$, as desired. This completes the proof.
\end{proof}


\end{document}